\documentclass[12pt,A4]{article}
\usepackage[english]{babel}
\usepackage{amsmath,amsfonts,amssymb,amsthm,mathrsfs,bbm}
\usepackage[font=sf, labelfont={sf,bf}, margin=1cm]{caption}
\usepackage{graphicx,graphics}
\usepackage{epsfig}
\usepackage{latexsym}
\usepackage[applemac]{inputenc}
\usepackage{ae,aecompl}
\usepackage{pstricks}
\usepackage{enumerate}
\usepackage{xcolor}
\usepackage[pdfpagemode=UseNone,bookmarksopen=false,colorlinks=true,urlcolor=blue,citecolor=blue,citebordercolor=blue,linkcolor=blue]{hyperref}

\usepackage[normalem]{ulem}
\usepackage[top=2.5cm,left=1.5cm,right=1.5cm,bottom=2.5cm]{geometry}
\usepackage{comment}
\usepackage{eulervm,palatino}

\usepackage[all, knot, cmtip]{xy} 

\newcommand{\ndN}{\mathbb{N}}
\newcommand{\ndZ}{\mathbb{Z}}

\newcommand{\ndR}{\mathbb{R}}

\renewcommand{\Pr}[1]{\mathbb{P}(#1)}
\newcommand{\Ex}[1]{\mathbb{E}[#1]}

\newcommand{\cC}{\mathcal{C}}

\newcommand{\cM}{\mathcal{M}}
\newcommand{\cF}{\mathcal{F}}
\newcommand{\cB}{\mathcal{B}}

\newcommand{\cG}{\mathcal{G}}
\newcommand{\cT}{\mathcal{T}}

\newcommand{\cA}{\mathcal{A}}
\newcommand{\cH}{\mathcal{H}}

\newcommand{\mA}{\mathsf{A}}

\newcommand{\mF}{\mathsf{F}}

\newcommand{\mK}{\mathsf{K}}

\newcommand{\mG}{\mathsf{G}}

\newcommand{\mS}{\mathsf{S}}
\newcommand{\mR}{\mathsf{R}}

\newcommand{\cE}{\mathcal{E}}

\newcommand{\eqdist}{\,{\buildrel d \over =}\,}

\newcommand{\Set}{\textsc{SET}}

\newcommand{\cZ}{\mathcal{Z}}

\newtheorem{theorem}{Theorem}[section]

\newtheorem{proposition}[theorem]{Proposition}
\newtheorem{lemma}[theorem]{Lemma}

\newtheorem{definition}[theorem]{Definition}

\numberwithin{equation}{section}

\title{\textbf{ Gibbs partitions: the convergent case}}
\date{}

\author{Benedikt Stufler\thanks{\'Ecole Normale Sup\'erieure de Lyon, E-mail: benedikt.stufler@ens-lyon.fr; The author is supported by the German Research Foundation DFG, STU 679/1-1}}

\begin{document}
	
	\maketitle
	
\let\thefootnote\relax\footnotetext{ 
\\\emph{MSC2010 subject classifications:} Primary  05A18,05C80 ; secondary 05C30. }

\vspace {-0.5cm}

\begin{abstract}
	We study Gibbs partitions that typically form a unique giant component. The remainder is shown to converge in total variation toward a Boltzmann-distributed limit structure.  We demonstrate how this setting encompasses arbitrary weighted assemblies of tree-like combinatorial structures. As an application, we establish smooth growth along lattices for small block-stable classes of graphs. Random graphs with $n$ vertices from such classes are shown to form a giant connected component. The small fragments may converge toward different Poisson Boltzmann limit graphs, depending along which lattice we let $n$ tend to infinity. Since proper addable minor-closed  classes of graphs belong to the more general family of small block-stable classes, this recovers and generalizes results by McDiarmid (2009).
\end{abstract}

\section{Introduction}

The motivation for the present work stems from various areas, starting with enumerative combinatorics. It was conjectured by Bernardi, Noy and Welsh~\cite{2007arXiv0710.2995B} that proper minor-closed addable classes have smooth growth. Such a condition crops up in related contexts, for example, in McDiarmid, Steger and Welsh \cite{MR2117936,MR2249274}. The conjecture was confirmed by McDiarmid \cite{MR2507738}. One of the methods used is an approach used in Bender, Canfield and Richmond \cite{MR2383441}, who proved smoothness for classes of graphs embeddable on any fixed surface. It was established in \cite{MR2507738} furthermore that random graphs from proper minor-closed addable classes typically admit a giant component, and that the remaining fragments converge in total variation toward a limit called the Boltzmann Poisson random graph of the class. Such a behaviour had previously been observed for random planar graphs by McDiarmid \cite{MR2418771}. The enumerative study of minor-closed classes and related classes of graphs has since then received growing attention in the literature, see Noy \cite{noysurvey} for a comprehensive survey.

A  link from this topic to general models of random partitions can be found in the work by Barbour and Granovsky \cite{MR2121024}. The authors use a perturbed Stein recursion approach to study the asymptotic behaviour of random partitions satisfying a conditioning relation. Various regimes with differing behaviour are known for this general model of partitions~\cite{MR2032426}, and the "convergent case" setting of \cite{MR2121024} is characterized by exhibiting a giant component, whose remainder converges toward an almost surely finite limit. It is an interesting observation, that the distribution of the remainder in \cite{MR2121024} belongs to a family encompassing  the components of the Boltzmann Poisson random graph constructed by McDiarmid \cite{MR2507738}, and only natural, to check whether these results on random partitions can be applied to random graphs. Clearly great care was taken in \cite{MR2121024} to use only a minimum set of requirements, but too little is known apart from smoothness about the asymptotic number of graphs in an arbitrary proper addable minor-closed classes. 

As it is not clear whether these results apply, other options have to be considered. Apart from random partitions that are characterized by a conditioning relation, there is another well-known model that encompasses the component distribution of random graphs from proper addable minor-closed classes: Gibbs partitions.  Both families of random structures are quite general and have a non-trivial intersection, but neither contains the other. The term was coined by Pitman \cite{MR2245368} in his comprehensive survey on combinatorial stochastic processes, and since then further important additions to the theory were made \cite{MR2453776}. Gourdon~\cite[Thm. 1]{MR1603725} gave results in a specific setting, where a giant component emerges, and the size of the remainder converges in distribution. He required the exponential generating function of the structures on the components to be amendable to singularity analysis~\cite[Thm. 1]{MR2483235}, such that its coefficients are asymptotically close to $c (\log n)^\beta n^{-\alpha}$ for some constants $c$, $\beta$ and  $1 < \alpha <2$. Whenever methods from analytic combinatorics apply, they yield results of great precision, which is impressively demonstrated in the tail-bounds \cite[Thm. 2]{MR2483235} for the size of the remainder. However, these requirements are much more specific than in the mentioned work \cite{MR2121024} for partitions with a conditioning relation, and there are known examples of minor-closed addable classes such as random planar graphs \cite{MR2476775}, for which $\alpha=7/2$ lies outside of the considered interval.

For these reasons, it is desirable to establish a "convergent case" regime for Gibbs partitions, that is as general as possible, and in which a similar behaviour as in Barbour and Granovsky's  setting~\cite{MR2121024} may be observed.  In the present work, we consider Gibbs partitions with a subcritical composition scheme, such that the generating series of the structures on the components belongs to the family of subexponential sequences studied in  \cite{MR0348393,MR714482,MR772907}. The elements of this family  correspond up to tilting and normalizing to subexponential densities of lattice distributed random variables, and hence may be put in the general context of heavy-tailed and subexponential distributions \cite{MR3097424}.  Our first main result establishes that Gibbs partitions in this setting exhibit a giant component, and the small rest converges in total variation toward a limit structure following a weighted Boltzmann distribution. In order to demonstrate its broad scope and relevance for combinatorial questions, we use the strong ratio property and a number of results related to simply generated trees \cite{MR2245498}, to show how analytic assemblies of arbitrarily tree-like combinatorial structures belong to this regime. 

We apply our results to small block-stable classes of graphs. For any such class $\cA$, we partition the integers into a finite set of shifted lattices of the form $a + d\ndZ$ for $0 \le a < d$, along which the class $\cA$  grows smoothly. This allows us to characterize precisely when $\cA$ belongs to the family of smooth graph classes. The uniform $n$-sized random graph from $\cA$ is shown to form a giant component with a stochastically bounded remainder.  The fragments not contained in the giant component converge to different Boltzmann Poisson random graphs,  depending along which lattice we let $n$ tend to infinity. Any proper addable minor-closed class of graphs is small and block-stable, but the converse does not hold. Hence this recovers and generalizes  corresponding results by McDiarmid \cite[Theorems 1.2, 1.7]{MR2507738}, who established smooth growth and convergence of the small fragments for proper addable minor-closed classes. Our approach also works in various other settings, for which we provide some examples, including random graphs drawn with probability proportional to weights assigned to their blocks. 

\subsection*{Plan of the paper}
In Section~\ref{sec:prel} we fix notations and recall necessary background related to Gibbs partitions, graph classes and subexponential sequences. Section~\ref{sec:conv} presents our results on Gibbs partitions in the convergent case, and Section~\ref{sec:graphs} discusses their applications to small block-stable graph classes. Section~\ref{sec:ext} discusses extensions to similar settings. In Section~\ref{sec:proofs} we collect all proofs.

\section{Preliminaries}
\label{sec:prel}

\subsection{Notation}

Throughout, we set
\[
\ndN=\{1,2,\ldots\}, \qquad \ndN_0 = \{0\} \cup \ndN, \qquad [n]=\{1,2,\ldots, n\}, \qquad n \in \ndN_0.
\]
We usually assume that all considered random variables are defined on a common probability space $(\Omega, \mathscr{F}, \mathbb{P})$.  All unspecified limits are taken as $n$ becomes large, possibly along an infinite subset of $\ndN$.  The {\em total variation distance} between two random variables $X$ and $Y$ with values in a countable state space $S$ is defined by
\[
d_{\textsc{TV}}(X,Y) = \sup_{\cE \subset S} |\Pr{X \in \cE} - \Pr{Y \in \cE}|.
\]
A sequence of $\ndR$-valued random variables $(X_n)_{n \ge 1}$ is {\em stochastically bounded}, if for each $\epsilon > 0$ there is a constant $M>0$ with
\[
\limsup_{n \to \infty} \Pr{ |X_n| \ge M} \le \epsilon.
\]
We let $\ndR_{>0}$ and $\ndR_{\ge 0}$ denote the sets of positive and non-negative real numbers, respectively. A function $h: \ndR_{>0} \to \ndR_{>0}$ is called {\em slowly varying}, if
\[
\lim_{x \to \infty}\frac{h(tx)}{h(x)} = 1
\]
for all fixed $t>0$. For any power series $f(z)$, we let $[z^n]f(z)$ denote the coefficient of $z^n$.

\subsection{Weighted combinatorial species and generating functions}
Let $\cF^\omega$ denote a {\em species of combinatorial structures}  with non-negative weights in the sense of Joyal~\cite{MR633783}.
That is, for each finite set $U$ we are given a finite set $\cF[U]$ of {\em $\cF$-structures} and a map
\[
\omega_U: \cF[U] \to \ndR_{\ge 0}.
\]
Moreover, for each bijection $\sigma: U \to V$ the species $\cF$ produces  a  corresponding bijection 
\[
\cF[\sigma]: \cF[U] \to \cF[V]
\]
that preserves the $\omega$-weights. This may be expressed by requiring that the diagram
\[
\xymatrix{ \cF[U]  \ar[r]^{\cF[\sigma]} \ar[dr]^{\omega_U} 
	&\cF[V]\ar[d]^{\omega_V}\\
 		    &\ndR_{\ge 0}}
\]
commutes. Species are also subject to the usual functoriality requirements: the identity map $\text{id}_U$ on $U$ gets mapped to the identity map $\cF[\text{id}_U] = \text{id}_{\cF[U]}$ on the set $\cF[U]$. For any bijections $\sigma: U \to V$ and $\tau: V \to W$ the diagram
\[
\xymatrix{ \cF[U]  \ar[r]^{\cF[\sigma]} \ar[dr]^{\cF[\tau \sigma]} 
	&\cF[V]\ar[d]^{\cF[\tau]}\\
	&\cF[W]}
\]
commutes. As a last requirement, we also assume that $\cF[U] \cap \cF[V] = \emptyset$ whenever $U \ne V$. This is not much of a restriction, as we may always replace $\cF[U]$ by $\{U\} \times \cF[U]$ for all sets $U$, to make sure that it is satisfied. 

Two weighted species $\cF^\omega$ and $\cH^\gamma$ are {\em structurally equivalent} or {\em isomorphic}, denoted by $\cF^\omega \simeq \cH^\gamma$, if there is a family of weight-preserving bijections $(\alpha_U: \cF[U] \to \cH[U])_U$ with $U$ ranging over all finite sets, such the following diagram commutes for each  bijection  $\sigma: U \to V$ of finite sets.
\[
\xymatrix{ \cF[U] \ar[d]^{\alpha_U} \ar[r]^{\cF[\sigma]} &\cF[V]\ar[d]^{\alpha_V}\\
	\cH[U] \ar[r]^{\cG[\sigma]} 		    &\cH[V]}
\]

We will often write $\omega(F)$ instead of $\omega_U(F)$ for the weight of a structure $F \in \cF[U]$. For any $\cF$-object $F \in \cF[U]$ we let
\[
|F| := |U| \in \ndN_0
\]
denote its {\em size}. It will be convenient to use the notation
\[
	\mathscr{U}(\cF) = \bigcup_{n \ge 0} \cF[n].
\]
 Here we write $\cF[n] = \cF[ \{1, \ldots, n \}]$ for all non-negative integers $n$. This allows us to define the {\em exponential generating series}
\[
	\cF^\omega(z) = \sum_{F \in \mathscr{U}(\cF)} \omega(F) \frac{z^{|F|}}{|F|!} 
\]
as a formal power series with non-negative coefficients. A simple example is the species $\Set$ where $\Set[U] = \{U\}$ for each finite set $U$ and each object receives weight $1$. Hence $\Set(z) = \exp(z)$.

Two $\cF$-objects $F_1 \in \cF[U]$ and $F_2 \in \cF[V]$ are termed {\em isomorphic}, denoted by $F_1 \simeq F_2$, if there is a bijection $\sigma:U \to V$ such that $\cF[\sigma](F_1) = F_2$. An {\em unlabelled} $\cF$-object is formally defined as a  maximal class of pairwise isomorphic objects. The unlabelled object corresponding to a given $\cF$-object $F$ is also termed its {\em isomorphism type} and denoted by $\tilde{F}$.  

We are going to consider probability measures on the collection of all unlabelled $\cF$-objects. From a formal viewpoint, this may be slightly problematic, because infinite collections of proper classes are not well-defined objects. But this more of a notational issue that could easily be resolved by working with a fixed set of representatives instead.

\subsection{Composite and derived structures}

 Let $\cF^\omega$ and $\cG^\nu$ be combinatorial species with non-negative weights, such that $\cG^\nu[\emptyset] = \emptyset$. The {\em composition} $\cF^\omega \circ \cG^\nu = (\cF \circ \cG)^\mu$ of the two species describes partitions of finite sets where each partition class is endowed with a $\cG$-structure, and the collection of partition classes carries an $\cF$-structure. Formally, it is defined by setting for each finite set $U$
\[
	(\cF \circ \cG)[U] = \bigcup_{\pi}  \cF[\pi] \times \prod_{Q \in \pi} \cG[Q]
\]
with the index $\pi$ ranging over all unordered partitions of $U$ with non-empty partition classes. That is, $\pi$ is a set of non-empty subsets of $U$ such that $U = \bigcup_{Q \in \pi} Q$ and  $Q \cap Q' = \emptyset$ for all $Q, Q' \in \pi$ with $Q \ne Q'$. The weight of a composite structure $(F, (G_Q)_{Q \in \pi})$ is defined by 
\[
	\mu(F, (G_Q)_{Q \in \pi}) = \omega(F) \prod_{Q \in \pi} \nu(G_Q).
\]
For any bijection $\sigma: U \to V$, the corresponding function 
\[
(\cF \circ \cG)[\sigma]: (\cF \circ \cG)[U] \to (\cF \circ \cG)[V]
\]
is defined as follows. For each element $(F, (G_Q)_{Q \in \pi}) \in (\cF \circ \cG)[U]$ we let $\bar{\pi} = \{ \sigma(Q) \mid Q \in \pi\}$ denote a partition of $V$ and set
\[
\bar{\sigma}: \pi \to \bar{\pi}, Q \mapsto \pi(Q).
\]
For each $Q \in \pi$ we let 
\[
\sigma|_Q: Q \to \sigma(Q), x \mapsto \sigma(x) \]
 denote the restriction of $\sigma$ to the class $Q$. We set
\[
(\cF \circ \cG)[\sigma](F, (G_Q)_{Q \in \pi}) = (\cF[\bar{\sigma}](F), ( \cG[ \sigma|_Q](G_{\sigma^{-1}(P)})_{ P \in \bar{\pi}}).
\]
The generating series of the composition satisfies \cite[Prop. 24]{MR633783}
\[
	(\cF^\omega \circ \cG^\nu)(z) = \cF^\omega(\cG^\nu(z)).
\]

A further construction that we are going to use  is the {\em derived} species $(\cF')^\omega$ defined as follows. For each set $U$ we let $*_U$ denote a placeholder object not contained in $U$. For example, we could define $*_U = U$, as no set is allowed to be an element of itself. We set
\[
	\cF'[U] = \cF[U \cup \{ *_U\}].
\]
The weight of an element $F' \in \cF'[U]$ is its $\omega$-weight as $\cF$-structure. Any bijection $\sigma: U \to V$ may canonically be extended to a bijection
\[
	\sigma': U \cup \{ *_U\} \to V \cup \{ *_V\},
\]
and we set
\[
	\cF'[\sigma] = \cF[\sigma'].
\]
Thus, an $\cF'$-object with size $n$ is an $\cF$-object with size $n+1$, since we do not count the $*$-placeholder. The exponential generating series of $(\cF')^\omega$ is given by the formal derivative
\[
	(\cF')^\omega(z) = \frac{\text{d}}{\text{d}z} \cF^\omega(z).
\]

\subsection{Kolchin's representation theorem and Boltzmann distributions}

Given a weighted species $\cF^\omega$ and a parameter $y>0$ with $0 < \cF^\omega(y) < \infty$, we may consider the corresponding Boltzmann probability measure
\[
	\mathbb{P}_{\cF^\omega, y}(F) = \cF^\omega(y)^{-1} y^{|F|}\omega(F)/|F|!, \qquad F \in \mathscr{U}(\cF).
\]

In a certain sense, Boltzmann measures are invariant under relabelling: 

\begin{proposition}
Let $\mF$ follow a $\mathbb{P}_{\cF^\omega, y}$ distribution and let $N = |\mF|$ denote its random size. If we draw a permutation $\sigma: [N] \to [N]$ uniformly at random, then the corresponding relabelled object $\cF[\sigma](\mF)$ is also $\mathbb{P}_{\cF^\omega, y}$  distributed.
\end{proposition}

The Boltzmann distribution for composite structures admits a useful canonical coupling, which is a combinatorial interpretation of Kolchin's representation theorem \cite[Thm. 1.2]{MR2245368}, and also known as the substitution rule for Boltzmann samplers. It is constructed in \cite{MR2810913} for species without weights, and the generalization to the weighted setting is straight-forward.

\begin{lemma}[\cite{MR2810913}]
	\label{le:composition}
	Let $\cF^\omega$ and $\cG^\nu$ be weighted species with $\cG[\emptyset] = \emptyset$. Let $x>0$ be a parameter with $0< \cF^\omega(\cG^\nu(x)) < \infty$ and $y := \cG^\nu(x)< \infty$. If we sample a $\mathbb{P}_{\cF^\omega, y}$-distributed $\cF$-object $\mF$, and for each $1 \le i \le |\mF|$ and independent $\mathbb{P}_{\cG^\nu, x}$-distributed $\cG$-object $\mG_i$, then the tupel $(\mF, \mG_1, \ldots, \mG_{|\mF|})$ may be interpreted as an $\cF \circ \cG$-object $\mS$ on the disjoint union
	\[
		V = \bigsqcup_{1 \le i \le |\mF_i|} [|\mG_i|].
	\]
	Let $\sigma: V \to [|V|]$ be a uniformly at random sampled bijection. Then the relabelled object
	\[
		(\cF \circ \cG)[\sigma](\mS) \in \mathscr{U}(\cF \circ \cG)
	\]
	follows a $\mathbb{P}_{\cF^\omega \circ \cG^\nu, x}$-distribution.
\end{lemma}

\subsection{Graph classes}
A {\em simple finite graphs} is a pair $
G = (V,E)
$
of a finite set $V = V(G)$ of {\em vertices} or {\em labels}  together with a set $E= E(G)$ of {\em edges} \[E \subset \{ \{x,y\} \mid x,y \in V, x \ne y\}.\]
To avoid notational ambiguities we assume additionally that $V \cap E = \emptyset$. Two vertices $x,y \in V$ are {\em adjacent}, if $\{x,y\} \in E$. We also say that $y$ is a {\em neighbour} of $x$. A {\em subgraph} of $G$ is a  graph $H$ with $V(H) \subset V(G)$. We say $H$ is a {\em proper subgraph}, if additionally $H \ne G$.

We term $G$ {\em connected}, if its vertex set is non-empty, and for all $x,y \in V$ we can reach $y$ by starting at $x$ and traversing edges.  A {\em connected component} of $G$ is a connected subgraph $H$ that is maximal with this property. That is, no other connected subgraph of $G$ exists that contains $H$ as a proper subgraph.  We say $G$ is $2$-connected, if $G$ is connected, has at least $3$ vertices, and deleting an arbitrary single vertex does not disconnect the graph. A subgraph $H$ of $G$ is termed a {\em block}, if it is either an isolated vertex with no neighbours, or two vertices joined by a single edges whose deletion would increase the number of connected components, or a $2$-connected subgraph that is maximal with this property. 

For each bijection
$
\sigma: V \to U
$
between the vertex set of $G$ and an arbitrary finite set $U$  we may form the {\em relabelled graph} 
\[
\sigma.G := (U, \{ \{\sigma(x), \sigma(y)\} \mid \{x,y\} \in E(G)\}.
\]
Two graphs are termed {\em isomorphic} if one is a relabelled version of the other.  For any edge $e = \{x,y\} \in E(G)$ we may form a new graph $G/e$ by {\em contracting}  $e$. The graph $G/e$ is formed by replacing $x$ and $y$ with a single vertex $v_{x,y}$ that is adjacent to all former neighbours of $x$ and $y$. A graph $H$ is a {\em minor} of a graph $G$, if there are graphs $G_0, \ldots, G_t$ such that $G_0 \simeq G$ and $G_t \simeq H$ and for each $G_{i+1}$ arises from $G_i$ by deleting an edge, contracting an edge, or deleting a vertex.

A collection $\cA$ of graphs that is closed under relabelling is termed a {\em graph class}. For notational convenience, we will always assume that $\cA$ contains the trivial graph whose vertex-set is the empty set. We may interpret $\cA$ as a combinatorial species by letting, for each finite set $V$, $\cA[V] \subset \cA$ denote the finite subset of all graphs in $\cA$ with vertex set $V$, and defining
\[
\cA[\sigma]: \cA[V] \to \cA[U], G \mapsto \sigma.G
\]
for each finite set $U$ and bijection $\sigma: V \to U$. Moreover, we assign weight $1$ to each $\cA$-object. 

We say a graph class $\cA$ is
\begin{enumerate}
	\item {\em proper}, if there exists a graph that is not contained in $\cA$.
	\item {\em small}, if the radius of convergence of the exponential generating series $\cA(z)$ is positive.
	\item {\em decomposable}, if any graph lies in $\cA$ if and only if all its connected components do.
	\item {\em bridge-addable}, if, for each graph $G \in \cA$ and each pair of vertices $x,y \in V(G)$ contained in different connected components of $G$, the graph obtained by adding the edge $\{x,y\}$ to $G$ also belongs to $\cA$.
	\item {\em addable}, if it is both decomposable and bridge-addable.
	\item {\em minor-closed}, if for each $G \in \cA$ and each minor $H$ of $G$ it also holds that $H \in A$.
	\item {\em block-stable}, if it contains the graph consisting of a single vertex, and any graph lies in $\cA$ if and only if all its blocks do.
	\item {\em smooth}, if its generating series $\cA(z)$ has a finite positive radius of convergence and satisfies the ratio test.
\end{enumerate}

If $\cA$ is decomposable and $\cC \subset \cA$ denotes the subclass of all connected graphs in $\cA$, then the two species are related by a canonical isomorphism
\begin{align}
\label{eq:setiso}
\cA \simeq \Set \circ \cC.
\end{align}
This expresses the fact the connected components of a graph form a partition of the vertex set of the graph, and any combination of connected graphs from $\cC$ must lie in $\cA$, since $\cA$ is decomposable. 

If the graph class $\cA$ is block-stable, then it is also decomposable. We may consider the subclass $\cB \subset \cC$ of all graphs in $\cC$ that are $2$-connected or consist of two vertices joined by a single edge. It was noted  by Harary and Palmer \cite[1.3.3, 8.7.1]{MR0357214},  Robinson \cite[Thm. 4]{MR0284380}, and Labelle \cite[2.10]{MR699986} that
\begin{align}
\label{eq:simple}
z \cC'(z) = z \phi(z \cC'(z))
\qquad \text{with} \qquad
\phi(z) = \exp(\cB'(z)).
\end{align}


It is well-known that all minor-closed classes are block-stable. To see this, suppose that $\cA$ is a minor-closed graph class. An {\em excluded minor} of $\cA$ is a graph that does not belong to $\cA$, but all its proper minors do. If $\cM$ denotes the collection of excluded minors of $\cA$, then a graph lies in $\cA$ if and only if none of its minors belongs to $\cM$. A moment's though verifies that $\cA$ is decomposable, if and only if all excluded minors are connected, and addable, if and only if all excluded minors are $2$-connected. Graph classes defined by excluding $2$-connected minors must be block-stable, because every $2$-connected subgraph of a graph $G$ is a subgraph of one of its blocks.

It was shown by Norine, Seymour, Thomas and Wollan \cite{MR2236510} that proper minor-closed classes are small. Thus proper addable minor-closed classes are examples of small block-stable classes containing all trees.

\subsection{Subexponential sequences}

We  consider power series whose coefficients belong to the family of subexponential sequences studied by Chover, Ney and Wainger \cite{MR0348393} and Embrechts \cite{MR714482}. Up to tilting and rescaling, these sequences corresond to subexponential densities of random variables with values in a lattice. Hence they may be put into the more general context of heavy-tailed and subexponential distributions, for which a comprehensive treatment is given in the book by Foss, Korshunov, and Zachary \cite{MR3097424}.  

\begin{definition}
	Let $d \ge 1$ be an integer. A power series $g(z) = \sum_{n =0}^\infty g_n z^n$ with non-negative coefficients and radius of convergence $\rho >0$ belongs to the class $\mathscr{S}_d$, if $g_n=0$ whenever $n$ is not divisible by $d$, and
	\begin{align}
	\label{eq:condition}
	\frac{g_n}{g_{n+d}} \sim \rho^d, \qquad \frac{1}{g_n}\sum_{i+j=n}g_ig_j \sim 2 g(\rho) < \infty
	\end{align}
	as $n \equiv 0 \mod d$ becomes large.
\end{definition}
The following theorem describes the behaviour of randomly stopped sums.

\begin{theorem}[{\cite[Thm. 4.8, 4.30]{MR3097424}}]
	\label{te:haupt}
	If $g(z)$ belongs to $\mathscr{S}_d$ with radius of convergence $\rho$, and $f(z)$ is a non-constant power series with non-negative coefficients that is analytic at $\rho$, then $f(g(z))$ belongs to $\mathscr{S}_d$ and
	\[
		[z^n] f(g(z)) \sim f'(g(\rho)) [z^n]g(z), \qquad n \to \infty, \qquad n \equiv 0 \mod d.
	\]
\end{theorem}

The broad scope of this setting is illustrated by the following easy observation, which has been noted in various places, see for example \cite{MR772907}.
\begin{proposition}
	\label{pro:easy}
 If $g_n = h(n) n^{-\beta} \rho^{-n}$ for some constants $\rho>0$, $ \beta > 1$ and a slowly varying function $h$, then the series $\sum_{n \in d\ndN} g_n z^n$ belongs to the class $\mathscr{S}_d$.
\end{proposition}

We will make use of the following criterion related to sums of random variables.
\begin{lemma}[{\cite[Thm. 4.9]{MR3097424}}]
	\label{le:help}
	Let $f(z)$ belong to $\mathscr{S}_1$ with radius of convergence $\rho$, and $g_1(z), g_2(z)$ be power-series with non-negative coefficients. If
	\[
		\frac{[z^n]g_1(z)}{[z^n]f(z)} \to c_1, \qquad \text{and} \qquad  \frac{[z^n]g_2(z)}{[z^n]f(z)} \to c_2
	\]
	as $n \to \infty$ with $c_1,c_2 \ge 0$, then
	\[
		 \frac{[z^n]g_1(z)g_2(z)}{[z^n]f(z)} \to c_1g_2(\rho) + c_2g_1(\rho).
	\]
	If additionally $c_1g_2(\rho) + c_2g_1(\rho)>0$, then $g_1(z)g_2(z)$ belongs to $\mathscr{S}_1$.
\end{lemma}

\section{Convergent Gibbs partitions}
\label{sec:conv}
Suppose that we are given combinatorial species $\cF^\omega$ and $\cG^\nu$ with $\cG[\emptyset] = \emptyset$ and $[z^k]\cF^\omega(z) >0$ for at least one $k \ge 1$. For each integer $n \ge 0$ with $[z^n] (\cF^\omega \circ \cG^\nu)(z) > 0$ we may sample a random composite structure 
\[
\mS_n = (\mF_n, (\mG_Q)_{Q \in \pi_n})
\] from the set $(\cF \circ \cG)[n]$ with probability proportional to its weight. The corresponding random partition $\pi_n$ of the set $[n]$  is termed a {\em Gibbs partition}.  We assume throughout that $(\cF^\omega \circ \cG^\nu)(z)$ is not a polynomial, so that we may  study $\mS_n$ as $n$ tends to infinity.

We are interested in the behaviour of the remainder $\mR_n$ when deleting "the" largest component from $\mS_n$.  More specifically, we construct $\mR_n$ as follows. We make a uniform choice of a component $Q_0 \in \pi_n$ having maximal size, and let $\mF_n'$ denote the $\cF'$-object obtained from the $\cF$-object $\mF_n$ by relabeling the $Q_0$ atom of $\mF_n$ to a $*$-placeholder. In more formal words, we set \[\mF_n' = \cF[\gamma](\mF_n) \in \cF'[\pi_n \setminus \{Q_0\}]\] for the  bijection $\gamma: \pi_n \to (\pi_n \setminus \{Q_0\}) \cup \{*\}$ with $\gamma(Q_0)=*$ and $\gamma(Q) = Q$ for $Q \ne Q_0$. This yields an $\cF' \circ \cG$-object
\[
	(\mF_n', (\mG_Q)_{Q \in \pi_n \setminus \{Q_0\}}) \in (\cF' \circ \cG)[[n] \setminus Q_0].
\]
We are not interested in the precise content of the underlying set $[n] \setminus Q_0$. Any $(n - |Q_0|)$-sized subset of $[n]$ is equally likely.  Hence we define the unique order-preserving map 
\[
\sigma: [n] \setminus \{Q_0\} \to [n - |Q_0|]
\]
and set
\begin{align*}
	\mR_n = (\cF' \circ \cG)[\sigma](\mF_n', (\mG_Q)_{Q \in \pi_n \setminus \{Q_0\}}) \in \mathscr{U}(\cF' \circ \cG).
\end{align*}
Alternatively, we could have chosen $\sigma$ uniformly at random, it wouldn't have changed the {\em distribution} of the outcome. We let $\mu$ denote the weighting on $ (\cF')^\omega \circ \cG^\nu$, that is, 
\begin{align*}
	\mu(F, (G_Q)_Q) = \omega(F) \prod_Q \nu(Q), \qquad (F, (G_Q)_Q) \in \mathscr{U}(\cF' \circ \cG).
\end{align*}

In a very general setting the remainder $\mR_n$ converges in total variation toward a limit object  following a Boltzmann distribution. 

\begin{theorem}
	\label{te:main}
	Suppose that the power series $\cG^\nu(z)$ belongs to the class $\mathscr{S}_d$ with radius of convergence $\rho$, and that $\cF^\omega(z)$ has radius of convergence strictly larger than $\cG^\nu(\rho)$. Let 
	$\mR$ be a random element of the set $\mathscr{U}(\cF' \circ \cG)$ that follows a Boltzmann distribution
	\[
	\Pr{\mR = R} = \mu(R) \frac{\rho^{|R|}}{|R|!} ((\cF')^\omega \circ  \cG^\nu)(\rho))^{-1}, \qquad R \in \mathscr{U}(\cF' \circ \cG).
	\]
	Then
	 \begin{align}
	 \label{eq:convergence}
	 d_{\textsc{TV}}(\mR_n, \mR) \to 0,  \qquad n\to \infty, \qquad   n \equiv 0 \mod d.
	 \end{align}
\end{theorem}

This implies convergence in total variation of the number of components, which has also been studied in \cite{MR1179830,MR1763972} for the case  $\cF^\omega = \Set$. We may also verify convergence of moments.

\begin{proposition}
	\label{pro:fftt}
	Suppose that the assumptions of Theorem~\ref{te:main} hold. Let $c(\cdot)$ denote the number of components in a composite structure. Then $c(\mS_n)$ converges towards $1 + c(\mR)$ in total variation and arbitrarily high moments.
\end{proposition}

Roughly speaking, the following lemma shows that Theorem~\ref{te:main} applies, whenever the species  $\cG^\nu$ is related to structures admitting a tree-like decomposition. This encompasses important families of enumerative series in combinatorics, for which  it is not known whether all members fall into the setting of Proposition~\ref{pro:easy}.   We demonstrate its usefulness in Section~\ref{sec:graphs} with a novel application to small block-stable graph classes. The proof of Lemma~\ref{le:simply} uses the strong ratio property and a variety of results related to simply generated trees. 

\begin{lemma}
	\label{le:simply}
		Let $\cZ(z) = \sum_{n \ge 1} Z_n z^n$ and $\phi(z) = \sum_{k \ge 0} \omega_k z^k$ be power series with non-negative coefficients that are related by the equation
		\[
		\cZ(z) = z \phi(\cZ(z)).
		\]
		Suppose that $\omega_0>0$, $\omega_k>0$ for at least one $k \ge 2$, and let $d$ denote the greatest common divisor of all $k$ with $\omega_k>0$. By \cite[Lem. 13.3]{MR2908619}, we know that $Z_n=0$ if $n-1$ is not a multiple of $d$, and $Z_n>0$ if $n \equiv 1 \mod d$ is large enough.
		Suppose that $\cZ(z)$ has non-zero radius of convergence $\rho_\cZ$. Then
		\begin{align*}
			Z_{n}^{-1} \sum_{i+j=n+1} Z_i Z_j \sim 2 \cZ(\rho_\cZ) / \rho_\cZ, \qquad n \to \infty, \qquad n \equiv 1 \mod d.
		\end{align*}
		This implies that the shifted series $\cZ(z)/z$ belongs to $\mathscr{S}_d$, since it was shown in \cite[Rem. 7.5]{MR2908619} that $\cZ(\rho_\cZ)<\infty$, and in \cite[Thm. 18.6, 18.10]{MR2908619}, that 
		\[
			Z_n/Z_{n+d} \to \rho_\cZ^d \qquad \text{and} \qquad Z_n^{1/n} \to 1/\rho_\cZ
		\]
		as $n \equiv 1 \mod d$ becomes large.
\end{lemma}

If the series $\cG^\nu(z)$ is periodic with a shift, then different behaviour may occur depending along which lattice we let $n$ tend to infinity.

\begin{theorem}
	\label{te:extension}
	Suppose that there is an integer $0 \le m < d$ such that $\cG^\nu(z) / z^m$ belongs to the class $\mathscr{S}_d$. Let $D = d/\gcd(m,d)$ and for each $0 \le a < D$, let  $\cF_a^\omega$ denote the restriction of $\cF^\omega$ to objects whose  size lies in $a + D\ndZ$. If the exponential generating series $\cF_a^\omega(z)$ is not constant, then
	\[
		d_{\textsc{TV}}(\mR_n, \mR(a)) \to 0, \qquad n \to \infty, \qquad n \equiv am \mod d
	\]
	with the limit object $\mR(a)$ following a $\mathbb{P}_{(\cF'_a)^\omega \circ \cG^\nu, \rho}$ Boltzmann distribution.
\end{theorem}

\section{Applications to random graphs}
\label{sec:graphs}

In the following, we let $\cA$ denote a small block-stable class of graphs,  $\cC \subset \cA$ its subclass of connected graphs, and $\cB \subset \cA$ the subclass of all graphs that are $2$-connected or consist of two vertices joined by a single edge. To exclude the case where $\cA$ is the trivial class of all graphs consisting of isolated points, we assume that $\cB$ is non-empty. Hence we may let $d$ denote the greatest common divisor of all integers $k$ with $[z^k] \exp(\cB'(z))>0$.  We let $\rho>0$ denote the radius of convergence of the exponential generating series $\cA(z) = \exp(\cC(z))$. Clearly $\rho$ is also the radius of convergence $\cC(z)$.

Using Lemma~\ref{le:simply} and the robustness of subexponential sequences against perturbation, we deduce the following enumerative result.

\begin{theorem}
	\label{te:app1}
	The series $\cC'(z)$ and $\cC(z)/z$ both lie in the class $\mathscr{S}_d$ of subexponential sequences with span $d$. That is $\cC'(\rho), \cC(\rho)< \infty$, and the coefficients $c_n = [z^n]\cC(z)$ satisfy
	\[
		\frac{c_n}{c_{n+d}} \sim \rho^d, \qquad \frac{1}{c_n} \sum_{i+j = n+1} c_i c_j \sim 2 \cC(\rho)/\rho
	\]
	as $n \equiv 1 \mod d$ becomes large.
\end{theorem}

For each integer $i$ we let $\Set_i$ denote the restriction of the species $\Set$ to objects whose size lies in the lattice $i + d \ndZ$. For each $0 \le a < d$ we let $\mG_a$ denote a random graph from the class $\cA$ following a $\mathbb{P}_{\Set_a' \circ \cC, \rho}$ Boltzmann distribution. That is, $\mR_a$ is a random graph from the class $\cA$ whose number of components lie in the shifted lattice $a-1 + d \ndZ$, and its distribution is given by
\[
	\Pr{\mG_a = G} = \frac{\rho^{|G|}}{|G|!} \left (\sum_{\substack{k \ge 0 \\ k \in a-1 + \ndZ}} \frac{\cC(\rho)^k}{k!}\right )^{-1}, \qquad G \in \mathscr{U}(\Set_{a-1} \circ \cC).
\]
For each integer $n \in \ndN_0$ let $\mA_n$ denote the random graph sampled uniformly from the set $\cA[n]$ of graphs in $\cA$ with vertex set $[n]$. We let $\textup{frag}(\mA_n)$ denote the graph obtained by deleting a uniformly at random drawn largest component of $\mA_n$, and relabelling the rest in a canonical order-preserving way to a set of the form $[k]$ for some $k\ge 0$. Alternatively, we may relabel by choosing a bijection uniformly at random, it makes no difference for the resulting {\em distribution}. Theorem~\ref{te:extension} yields our main application.
\begin{theorem}
	\label{te:app2}
	For each $0 \le a < d$, it holds that
	\[
		d_{\textsc{TV}}(\textup{frag}(\mA_n), \mG_a) \to 0
	\]
	as $n \equiv a \mod d$ becomes large. The coefficients of $\cA(z)$ along the lattice $a + d\ndZ$ belong, after a shift by $-a$, to the class $\mathscr{S}_d$ of subexponential sequences with span $d$. As $n \equiv a \mod d$ becomes large, it holds that
	\[
		[z^n] \cA(z) \sim C_{a-1} [z^{n+1-a}] \cC(z) \qquad \text{with} \qquad C_{a-1} = \sum_{\substack{k \ge 0 \\ k \equiv a-1 \mod d}} \cC(\rho)^k/k!.
	\]
\end{theorem}

This allows us to precisely describe under which conditions the graph class $\cA$ is smooth.
\begin{theorem}
	The graph class $\cA$ is smooth, if and only if  $d=1$.
\end{theorem}
Indeed, for $d=1$ the coefficients of $\cA(z)$ behave asymptotically up to a constant factor like those of $\cC(z)$, and hence grow smoothly. For $d \ge 2$, the only way for $\cA$ to be smooth is when $C_0 = \ldots = C_{d-1}$. There is a beautiful reason, why this may never happen. If $C_i = C_{i+1}$ would hold for all $i$, then we could select a $d$-th root of unity $\zeta \ne 1$ and deduce the contradiction
\begin{align*}
	\exp(\zeta\cC(\rho)) = C_0 + C_1 \zeta + \ldots + C_{d-1}  \zeta^{d-1} = C_0(1 + \zeta + \ldots + \zeta^{d-1}) 
	= 0.
\end{align*}
Hence $\cA$ cannot be smooth for $d \ge 2$.

\section{Extensions}
\label{sec:ext}

Many classes of weighted combinatorial composite structures may be expressed by a subcritical substitution scheme
$
		\cF^\omega \circ \cG^\nu,
$
such that Lemma~\ref{le:simply} may be applied either directly, or similarly as in the proof of Theorem~\ref{te:app1}, to show that $\cG^\nu$ belongs up to a constant shift to the class $\mathscr{S}_d$ for some $d \ge 1$. 
We illustrate this with some examples. There are of course many more, but we do not aim to provide an exhaustive list.

\paragraph{Random graphs with block-weights.}
	Random graphs from block-stable classes have a natural generalization to the weighted setting.
	Suppose that we are given a weighting $\gamma$ on the species $\cB$  of all graphs that are two-connected or consist of two vertices joined by a single edge. This yields a weighting on the species of connected graphs $\cC$, given by
	\[
		\nu(C) = \prod_B \gamma(B),
	\]
	with the index $B$ ranging over all blocks of the connected graph $C$. Here we set $\nu( \bullet)=1$ for the graph~"$\bullet$" consisting of a single vertex. Likewise we may define a weighting $\mu$ on the species $\cG$ of all graphs in the same way, such that
	\[
		\cG^\mu \simeq \Set \circ \cC^\nu.
	\]
	
	We may consider a random $n$-vertex graph $\mG_n^\mu$ drawn from $\cG^\mu[n]$ with probability proportional to its weight. This encompasses uniform random graphs from block-stable classes, which correspond precisely to the case where $\gamma(B) \in \{0,1\}$ for all blocks $B$. 
	
	The block-decomposition of connected graphs into $2$-connected components described for example in  Labelle \cite[2.10]{MR699986} can easily be seen to preserve the weights, yielding a weighted version of Equation~\eqref{eq:simple}:
	\[
		z (\cC')^\nu(z) = z \exp( (\cB')^\gamma(z (\cC')^\nu(z)).
	\]
	Hence a straight-forward analogon of Theorem~\ref{te:app2} also holds for the random graph $\mG_n^\mu$. The corresponding proof requires no modification at all.

\paragraph{Forests of Galton--Watson trees with a random number of trees.}
Let $(\cT_i)_{i \ge 1}$ be a family of independent copies of a subcritical or critical  Galton--Watson tree $\cT$ with offspring distribution $\xi$. Let $\mK$ denote an independent random non-negative integer having finite exponential moments. We may consider a  Galton--Watson forest $\mF$ with a random number of trees
\[
	\mF = (\cT_{1}, \ldots, \cT_K),
\]
and let
\[
	|\mF| = \sum_{i=1}^K |\cT|_i
\]
denote its size. The probability generating functions
\[
f(z) = \Ex{z^{|\mF|}}, \qquad \psi(z) = \Ex{z^K}, \qquad \phi(z) = \Ex{z^\xi}, \qquad \text{and} \qquad \cZ(z) = \Ex{z^{|\cT|}}
\]
are related by
\[
	f(z) = \psi(\cZ(z)) \qquad \text{and} \qquad \cZ(z) = z \phi(\cZ(z)).
\]
Obviously Lemma~\ref{le:simply} applies to $\cZ(z)$, so we obtain an analogon of Theorem~\ref{te:app2} that describes the asymptotic behaviour if we condition the forest $\mF$  to be large and cut down the largest tree.

\section{Proofs}
\label{sec:proofs}

We list the proofs of our results in order of their appearance.

\subsection{Proofs from Section~\ref{sec:conv}}

\begin{proof}[Proof of Theorem \ref{te:main}]
	In the following, we assume tacitly that  $n$ is divisible by $d$, and large enough such that $[z^n](\cF^\omega \circ \cG^\nu)(z) > 0$. 	The distributions of $\mR_n$ and $\mR$ are both invariant under relabelling uniformly at random. This implies that conditioned on having a fixed isomorphism type $\mathscr{T}$ of an $\cF' \circ \cG$-object with positive $\mu$-weight, each possible labelling of $\mathscr{T}$ is equally likely. In particular, it follows that
	\[ (\mR_n \mid \tilde{\mR} = \mathscr{T}) = (\mR \mid \tilde{\mR} = \mathscr{T}).\]
	Hence it suffices to establish total variational convergence of the isomorphism type of $\mR_n$ to the isomorphism type of $\mR$, that is
	\begin{align}
	\label{eq:showmee}
	\lim_{n \to \infty} d_{\textsc{TV}}(\tilde{\mR}_n, \tilde{\mR}) = 0.
	\end{align}
	
	Let $\mF$ denote a random $\cF$-object following a $\mathbb{P}_{\cF^\omega, \cG^\nu(\rho)}$-distribution, and for each $1 \le i \le |\mF|$ let $\mG_i$ be an independent $\mathbb{P}_{\cG^\nu, \rho}$ distributed $\cG$-object. Lemma~\ref{le:composition} states that, up to relabelling uniformly at random, the tupel \[\mS := (\mF, \mG_1, \ldots, \mG_{|\mF|})\] follows a $\mathbb{P}_{\cF^\omega \circ \cG^\nu, \rho}$ distribution. Consequently, if we condition $\mS$ on having size $n$, then it corresponds to an element from $(\cF \circ \cG)[n]$ that is sampled with probability proportional to its weight. That is, setting $f = |\mF|$ and $g_i = |\mG_i|$ for all $i$, it follows that as unlabelled objects 
	\[
	\mS_n \eqdist ( \mS \mid g_1 + \ldots + g_f =n).
	\]
	
	
	The random $\cF'\circ \cG$-object $\mR$ follows a $\mathbb{P}_{(\cF)'^\omega \circ \cG^\nu, \rho}$ distribution. Hence we may apply Lemma~\ref{le:composition} again to identify it up to relabelling  with a tuple
	\[
	\mR = (\mF', \mG^1, \ldots, \mG^{|\mF'|}),
	\]
	where $\mF'$ follows a $\mathbb{P}_{(\cF')^\omega, \cG^\nu(\rho)}$ distribution, and the $\mG^i$ are independent and $\mathbb{P}_{\cG^\nu, \rho}$ distributed. To simplify notation, we set $f' := |\mF'|$ and $g^i = |\mG^i|$ for all $i$. 
	
	Let $\hat{\mS}_n$ denote the composite $\cF \circ \cG$-structure obtained by sampling a random $\cG$-structure $\mG^*$ from $\cG[n - |\mR|]$ with probability proportional to its $\nu$-weight, and assigning it to the $*$-vertex of $\mF'$. This is only well-defined if 
	\begin{align}
	\label{eq:mycond}
	n - |\mR|>0 \qquad \text{and} \qquad  [z^{n- |\mR|}] \cG^\nu(z) > 0,
	\end{align} otherwise we set $\hat{\mS}_n$ to some place-holder value. The probability for the event \eqref{eq:mycond} tends to $1$ as $n$ becomes large, since $|\mR|$ is almost surely finite and a multiple of $d$, and $[z^{kd}]\cG^\nu(z) >0$ for all sufficiently large $k$ by assumption.
	
	We are going to show that as unlabelled $\cF \circ \cG$-objects
	\begin{align}
	\label{eq:convtype}
	\lim_{n \to \infty} d_{\textsc{TV}}( \mS_n, \hat{\mS}_n) = 0.
	\end{align}
	This implies Equation~\eqref{eq:showmee}, since the probability, that $\mR$ is the largest component of $\hat{\mS}_n$, tends to $1$ as $n$ becomes large.

	
	Let $g$ denote a random variable that is distributed like the size of a random $\cG$-object with a $\mathbb{P}_{\cG^\nu, \rho}$ distribution. Since $\cG^\nu(z)$ belongs to $\mathscr{S}_d$, it holds that
	\[
	\Pr{g = n+d} \sim \Pr{g = n}, \qquad n \to \infty.
	\]
	Consequently, there is a sequence $t_n$ of non-negative integers such that $t_n \to \infty$ and
	\begin{align}
	\label{eq:tohold}
	\lim_{n \to \infty} \sup_{\substack{0 \le y \le t_n \\ y  \equiv 0 \mod d}}  | \Pr{g = n + y}/\Pr{g = n} -1| = 0.
	\end{align}
	Indeed, for each $\epsilon>0$ and $t>0$ there is a constant $N_{\epsilon, t} \ge 1$ such that for all $n \ge N$
	\[
	\sup_{\substack{0 \le y \le t \\ y  \equiv 0 \mod d}}  | \Pr{g = n + y}/\Pr{g = n} -1| \le \epsilon.
	\]
	Setting $t_n = 1$ for $n < N_{2,1/2}$, $t_n = 2$ for $ N_{2,1/2} \le x < N_{2,1/2} + N_{3,1/3}$, and in general $t_n = k$ for
	\[
	N_{2,1/2} + \ldots + N_{k,1/k} \le n \le  N_{2,1/2} + \ldots + N_{k+1,1/(k+1)},
	\]
	yields a sequence with the desired properties.

	Let $k, x_1, \ldots, x_k \ge 1$ be integers with $x_1 + \ldots + x_k = n$. If we condition on $f=k$ and $g_i=x_i$ for all $1 \le i \le k$, then $\mF$ gets drawn from $\cF[k]$ with probability proportional to its $\omega$-weight, and likewise $\mG_i$ gets drawn from $\cG[x_i]$ with probability proportional to its $\nu$-weight for all $i$. Conditioned on having size $k-1$, $\mF'$ gets drawn from $\cF[ [k-1] \cup \{*\}]$ with probability proportional to its $\omega$-weight. Thus, up to relabeling uniformly at random, \[
	(\mF' \mid f'=k-1) \eqdist (\mF \mid f=k).
	\]
	Since $x_1 + \ldots + x_k = n$, it follows that as unlabelled $\cF \circ \cG$-objects
	\begin{align}
	\label{eq:gibtmp}
	(\mS \mid f=k, g_i=x_i, 1 \le i \le k) \eqdist (\hat{\mS}_n \mid f'= k-1, g^i=x_i, 1 \le i \le k-1).
	\end{align}	
	For any sequence $\mathbf{y} = (y_1, \ldots, y_{k-1})$ of positive integers with $D(\mathbf{y}) := y_1 + \ldots, y_{k-1} < n$ set
	\[
	\sigma_n(\mathbf{y}) = \{ (y_1, \ldots, y_{j-1}, n - D(\mathbf{y}), y_j, \ldots, y_k) \mid 1 \le j \le k\}.
	\]
	In order for Equation~\eqref{eq:tohold} to hold, we may replace the sequence $t_n$ by any other sequence that tends to infinity more slowly. So without loss of generality, we may assume that $t_n < n/2$ for all  $n$, and set
	\[
	M_n := \{ (k, \mathbf{y}) \mid k \ge 1, \mathbf{y} \in \ndN^{k-1}, D(\mathbf{y}) \le t_n \}.
	\]
	We are going to verify that
	\begin{align}
	\label{eq:gibtoshow}
	\Pr{ f=k, (g_1, \ldots, g_k) \in \sigma_n(\mathbf{y}) \mid g_1 + \ldots + g_f = n} \sim \Pr{ f'=k-1, (g^1, \ldots, g^{k-1}) = \mathbf{y} }
	\end{align}
	uniformly for all $(k, \mathbf{y}) \in M_n$. Since $t_n < n/2$, it holds that given $g_1 + \ldots g_f = n$ and $f=k$,  the event $(g_1, \ldots, g_k) \in \sigma_n(\mathbf{y})$ corresponds to $k$ distinct outcomes, depending on the unique location for the maximum of the $g_i$. Each outcome  is equally likely, so the  left-hand side in \eqref{eq:gibtoshow} divided by the right-hand side  equals
	\[
	\frac{k \Pr{f=k} \Pr{g=n - D(\mathbf{y}) } }{\Pr{f'=k-1} \Pr{g_1 + \ldots + g_f = n}},
	\]
	with
	\[
	\frac{k \Pr{f=k}}{\Pr{f'=k-1}} = \frac{ (\cF')^\omega(\cG^\nu(\rho)) \cG^\nu(\rho)}{\cF^\omega(\cG^\nu(\rho))} = \Ex{f}.
	\]
	As $D(\mathbf{y}) \le t_n$, it follows by Equation~\eqref{eq:tohold} that uniformly for $(k, \mathbf{y}) \in M_n$
	\[
	\Pr{g = n - D(\mathbf{y})} \sim \Pr{g = n}.
	\] 
	By Theorem~\ref{te:haupt} it holds that
	\[
	\Pr{g_1 + \ldots + g_f = n} \sim \Ex{f} \Pr{g=n},
	\]
	and \eqref{eq:gibtoshow} follows.
	
	To complete the proof, we first note that \begin{align}
	\label{eq:toimb}
	\lim_{n \to \infty} \Pr{ (f'+1, (g^1, \ldots, g^{f'})) \in M_n} = 1.
	\end{align} Hence \eqref{eq:gibtoshow} yields that with probability tending to $1$ as $n$ becomes large
	\[
	((f, (g_1, \ldots, g_f)) \mid g_1 + \ldots g_f = n)  \in \{k\} \times \sigma_n(\mathbf{y}) \quad \text{for some} \quad (k, \mathbf{y}) \in M_n.
	\]
	Using Equation \eqref{eq:gibtmp}, it follows that uniformly for all sets $\cE$ of $n$-sized unlabelled $\cF \circ \cG$-objects 
	\begin{align*}
	\Pr{\mS_n \in \cE} &= \Pr{ \mS \in \cE \mid g_1 + \ldots g_f = n} \\ 
	&= o(1) + \sum_{ (k,\mathbf{y}) \in M_n} \Pr{ \mS \in \cE, f=k, (g_1, \ldots, g_f) \in \sigma_n(\mathbf{y}) \mid g_1 + \ldots g_f = n}.
	\end{align*}
	For each $(k, \mathbf{y}) \in  M_n$, the corresponding summand may be simplified to
	\begin{align*}
	\Pr{ \mS \in \cE, f=k, (g_1, \ldots, g_f) \in \sigma_n(\mathbf{y})} / \Pr{g_1 + \ldots g_f = n},
	\end{align*}
	and then expressed as the product
	\begin{align*}
	\Pr{ \mS \in \cE \mid f=k, (g_1, \ldots, g_k) \in \sigma_n(\mathbf{y})} \Pr{f=k, (g_1, \ldots, g_k)\in \sigma_n(\mathbf{y}) \mid g_1 + \ldots g_f = n}.
	\end{align*}	
	We treat the two factors separately. For the first, Equation~\eqref{eq:gibtmp} yields 
	\[
	\Pr{ \mS \in \cE \mid f=k, (g_1, \ldots, g_k) \in \sigma_n(\mathbf{y})} = \Pr{ \hat{\mS}_n \in \cE \mid f'= k-1, g^i=y_i, 1 \le i \le k-1}.
	\]
	By \eqref{eq:gibtoshow} it holds that
	\[
	\Pr{f=k, (g_1, \ldots, g_k)\in \sigma_n(\mathbf{y}) \mid g_1 + \ldots g_f = n} \sim \Pr{ f'=k-1, (g^1, \ldots, g^{k-1}) = \mathbf{y} }
	\]
	uniformly for all $(k, \mathbf{y}) \in M_n$. Using Equation \eqref{eq:toimb} it follows that
	\begin{align*}
	\Pr{\mS_n \in \cE} &= o(1) + \sum_{(k, \mathbf{y}) \in M_n}(1+ o(1)) \Pr{\hat{\mS}_n \in \cE , f' = k-1, (g^1, \ldots, g^{k-1}) = \mathbf{y}}\\
	&= o(1) + \Pr{\hat{\mS}_n \in \cE}.
	\end{align*}
	This completes the proof.	
\end{proof}

\begin{proof}[Proof of Proposition~\ref{pro:fftt}]
	By Theorem~\ref{te:main} we need only check the convergence of the moments. With $\cF^\omega(z) = \sum_{i=0}^\infty f_i z^i$, set $f(z) = \sum_{i=1}^\infty i^k f_i z^i$. Theorem~\ref{te:haupt} implies that
	\[
	\Ex{c(\mS_n)^k} = \frac{ [z^n] f(\cG^\nu(z))}{[z^n] \cF^\omega(\cG^\nu(z))} \sim \frac{f'(\cG^\nu(\rho))}{ (\cF')^\omega(\cG^\nu(\rho))} = \Ex{(c(\mR) +1)^k}.
	\]
\end{proof}

\begin{proof}[Proof of Lemma~\ref{le:simply}]
	We will tacitly assume that $n \equiv 1 \mod d$. The number $Z_n$ is known as the partition function of simply generated trees. That is, the random plane tree $\cT_n$ with distribution given by
	\[
	\Pr{\cT_n = T} = Z_n^{-1} \prod_{v \in T} \omega_{d^+(v)}
	\]
	for any plane tree $T$, with $d^+(v)$ denoting outdegree of a vertex $v$, that is, its number of sons.
	There is a well-known connection between simply generated trees and branching processes \cite{MR2908619,MR2484382}: Precisely when $\rho_\cZ>0$, there is a critical or subcritical Galton--Watson tree $\cT$, with
	\[
	\Pr{|\cT|=n} = Z_n \rho_\cZ^n / \cZ(\rho_\cZ),
	\]
	such the simply generated tree $\cT_n$ is distributed like the $\cT$ conditioned on having $|\cT|=n$ vertices:
	\[
	\cT_n \eqdist (\cT \mid |\cT|=n).
	\]
	Hence in order to verify
	\[
	Z_{n}^{-1} \sum_{i+j=n+1} Z_i Z_j \sim 2\cZ(\rho_\cZ)/\rho_\cZ,
	\] 
	we need to show that 
	\begin{align}
	\label{eq:toshow}
	\Pr{ |\cT| + |\cT'| = n+1} \sim 2 \Pr{|\cT|=n},
	\end{align}
	with $\cT'$ denoting an independent copy of $\cT$. Let $\xi$ with $\Ex{\xi} \le 1$ denote the offspring distribution of the Galton--Watson tree $\cT$. Let  $(\xi_i)_{i \ge 1}$  a family of independent copies of $\xi$, and 
	\[
	S_n = \xi_1 + \ldots + \xi_n
	\]
	the associated random-walk. Any list $d_1, \ldots, d_n$ in $\ndN_0$ corresponds to the out-degrees of a depth-first-search ordered list of vertices of a plane tree with size $n$, if and only if
	\[
	\sum_{i=1}^n d_i = n-1 \qquad \text{and} \qquad \sum_{i=1}^k d_i \ge k \quad \text{for all $k<n$}.
	\]
	A classical combinatorial observation, also called the cycle lemma, states that for any sequence $x_1, \ldots, x_{s} \ge -1$  of integers satisfying
	\[
	\sum_{i=1}^s x_i = -r
	\]
	for some $r \ge 1$, there are precisely $r$ integers $1 \le u \le s$ such that the cyclically shifted sequence \[x_i^{(u)} = x_{1 + (i+u)\mod s}\] satisfies
	\[
	\sum_{i=1}^\ell x_i^{(u)} > r
	\]
	for all $1 \le \ell \le s-1$;
	see for example \cite[Lem. 15.3]{MR2908619}.
	Hence
	\begin{align*}
	\Pr{|\cT|=n} &= \Pr{\xi_1 + \ldots + \xi_n = n-1, \xi_1 + \ldots + \xi_k \ge k \text{ for $k<n$}} \\
	&= \frac{1}{n} \Pr{S_n = n-1}.
	\end{align*}
	Likewise, any list $d_1, \ldots, d_{n+1}$ in $\ndN_0$ corresponds to the concatenation of the depth-first-search ordered lists of outdegrees of two plane trees with total size $n+1$, if and only if
	\[
	\sum_{i=1}^{n+1} d_i = n-1 \qquad \text{and} \qquad \sum_{i=1}^k d_i \ge k-1 \quad \text{for all $k\le n$}.
	\]
	This yields
	\begin{align*}
	\Pr{|\cT| + |\cT'| =n+1} &= \Pr{\xi_1 + \ldots + \xi_{n+1} = n-1, \xi_1 + \ldots + \xi_k \ge k -1 \text{ for $k \le n$}} \\&= \frac{2}{n+1} \Pr{S_{n+1} = n-1}.
	\end{align*}
	Hence 
	\[
	\frac{\Pr{|\cT|+ |\cT'|=n+1}}{\Pr{|\cT|=n}} = 2 \frac{n+1}{n} \frac{\Pr{S_{n+1} = n-1}}{\Pr{S_n = n-1}}.
	\]
	By the strong ratio property~\cite{MR0107292}, it holds that
	\[
	\Pr{S_{n+1} = n-1} \sim \Pr{S_n = n-1}.
	\]
	This verifies \eqref{eq:toshow} and completes the proof.
\end{proof}

\begin{proof}[Proof of Theorem~\ref{te:extension}]
	Let $\mF$ denote a random $\cF$-object following a $\mathbb{P}_{\cF^\omega, \cG^\nu(\rho)}$-distribution, and for each $1 \le i \le |\mF|$ let $\mG_i$ be an independent $\mathbb{P}_{\cG^\nu, \rho}$ distributed $\cG$-object. Then Lemma~\ref{le:composition} yields that the tupel \[\mS := (\mF, \mG_1, \ldots, \mG_{|\mF|})\] follows up to relabelling a $\mathbb{P}_{\cF^\omega \circ \cG^\nu, \rho}$ distribution. We set $f =|\mF|$ and $g_i = |\mG_i|$ for all $i$. The Gibbs partition $\mS_n$ is a random structure sampled from  $(\cF \circ \cG)[n]$ with probability proportional to its weight. Hence it is distributed like the Boltzmann structure $\mS$ conditioned on having size $n$:
	\begin{align}
	\label{eq:t1}
	\mS_n \eqdist ( \mS \mid g_1 + \ldots + g_f =n).
	\end{align}
	By assumption, there is an integer $0 \le m < d$ such that  $\cG^\nu(z) / z^m$ lies in the class $\mathscr{S}_d$. In particular, we have 
	\[
		g_i \equiv m \mod d.
	\] for all $1 \le i \le f$. Recall that $D = d/ \gcd(m,d).$
	If $g_1 + \ldots + g_f = n$, then for all $0 \le a <D$ it holds that 
	\begin{align}
	\label{eq:t2}
	f \equiv a \mod D \qquad \text{if and only if} \qquad n \equiv am \mod d.
	\end{align}
	For $n \ge 1$, the event $f \equiv a \mod D$ has positive probability if and only if the restriction $\cF_a^\omega$ to objects with size in $a + D\ndZ$ has a non-constant generating function $\cF_a^\omega(z)$. Let us fix an integer $0 \le a <D$ with this property, and suppose that $n \equiv am \mod d$. Then Equations~\eqref{eq:t1} and \eqref{eq:t2} imply that
	\[
		\mS_n \eqdist (\mS \mid  g_1 + \ldots g_{f} = n, f \equiv a \mod D).
	\]
	Conditioned on having size in $a + D\ndZ$, the random $\cF$-object $\mF$ follows a $\mathbb{P}_{\cF_a^\omega \circ \cG^\nu,\rho}$ distribution. Let $\mF_a$ denote a $\mathbb{P}_{\cF_a^\omega, \cG^\nu(\rho)}$ distributed $\cF_a$-object that is independent from all previously considered random variables. It follows that
	\[
		\mS_n \eqdist ( (\mF_a, \mG_1, \ldots, \mG_{|\mF_a|}) \mid g_1 + \ldots + g_{|\mF_a|} = n).
	\]
	It follows by Lemma~\ref{le:composition}, that the vector
	\[
		(\mF_a, \mG_1, \ldots, \mG_{|\mF_a|})
	\]
	has a $\mathbb{P}_{\cF_a^\omega \circ \cG^\nu,\rho}$ Boltzmann distribution, and consequently $\mS_n$ is distributed like the $n$-sized Gibbs partition for $\cF_a^\omega \circ \cG^\nu$. If we can verify that 
	\begin{align}
		\label{eq:refme}
		\Pr{ g_1 + \ldots + g_{f_a} = n} \sim \Ex{f_a} \Pr{g = n - (a-1)m}, \qquad n \to \infty, \qquad n \equiv am \mod d,
	\end{align}
	then the convergence in total variation of $\mR_n$ toward $\mR(a)$ follows in an entirely analogous manner as in the proof of Theorem~\ref{te:main}. 

	Thus it remains to check \eqref{eq:refme}. Let $g$ be distributed according to the size of a random $\cG$-object following a $\mathbb{P}_{\cG^\nu, \rho}$ Boltzmann distribution. We may write
	\[
		g = m + d \bar{g}, \qquad f_a = a + \bar{f}_aD, \qquad n = am + \bar{n}d,
	\]
	with $\bar{n} \in \ndN_0$, and $\bar{g}, \bar{f}_a$ random non-negative integers. We let $(\bar{g}_i^{(j)})_{i,j \ge 0}$, denote independent copies of $\bar{g}$, and set \[
		S_i^{(j)} = g_1^{(j)} + \ldots + g_i^{(j)}.
	\]	
	Thus
	\[
		\Pr{g_1 + \ldots + g_{f_a} = n} = \Pr{ S_a^{(0)} + (S_D^{(1)} + Dm/d) + \ldots + (S_D^{(\bar{f}_a)} + Dm/d) = \bar{n}}.
	\]
	Since $\cG^\nu(z) / z^m$ lies in the class $\mathscr{S}_d$ by assumption, it follows that the probability weight sequence of $\bar{g}$ lies in $\mathscr{S}_1$. By Theorem~\ref{te:haupt} it follows that the densities of $S_a^{(0)}$ and $S_D^{(j)} + Dm/d$ belong to $\mathscr{S}_1$. Applying Theorem~\ref{te:haupt} again yields that the same holds for the randomly stopped sum $\sum_{j=1}^{\bar{f}_a}(S_D^{(j)} + Dm/d)$, with
	\[
		\Pr{ S_a^{(0)} = \bar{n}} \sim a \Pr{\bar{g}=\bar{n}} \qquad \text{and} \qquad \Pr{\sum_{j=1}^{\bar{f}_a}(S_D^{(j)} + Dm/d) =x} \sim D \Ex{\bar{f}_a}  \Pr{\bar{g}=\bar{n}}
	\]
	as $\bar{n} \to \infty$. Hence Lemma~\ref{le:help} yields
	\[
		\Pr{ S_a^{(0)} + (S_D^{(1)} + Dm/d) + \ldots + (S_D^{(\bar{f}_a)} + Dm/d) = \bar{n}} \sim (a + D\Ex{\bar{f}_a}) \Pr{\bar{g}=\bar{n}}.
	\]
	This verifies \eqref{eq:refme} and thus completes the proof.	
\end{proof}

\subsection{Proofs from Section~\ref{sec:graphs}}

\begin{proof}[Proof of Theorem~\ref{te:app1}]
	Equation~\eqref{eq:simple} yields that
	\[
	z\cC'(z) = z \phi(z\cC'(z))
	\]
	for the power series $\phi(z) = \exp(\cB'(z))$. Hence we may apply Lemma~\ref{le:simply} and obtain that  the series $\cC'(z)$ belongs to the class $\mathscr{S}_d$. That is, the coefficients $x_n = [z^n] \cC'(z) = (n+1) c_{n+1}$ satisfy
	\begin{align*}
	\cC'(\rho) < \infty, \qquad \frac{x_n}{x_{n+d}} \sim \rho^d, \qquad \frac{1}{x_n}\sum_{i+j=n}x_i x_j \sim 2 \cC'(\rho) < \infty,
	\end{align*}
	as $n \equiv 0 \mod d$ becomes large.
	It is clear, that this also implies
	\[
	\cC(\rho) < \infty \qquad \text{and} \qquad \frac{c_n}{c_{n+d}} \sim \rho^d, \qquad n \to \infty, \qquad n \equiv 1 \mod d.
	\]
	A characterization of subexponential series given for example in Foss, Korshunov, Zachary \cite[Thm. 4.21]{MR3097424} states that for any sequence $k_n \to \infty$ with $k_n < n/2$ it holds that
	\[
	\frac{1}{x_n} \sum_{\substack{i+j=n\\i,j\ge k_n}} x_i x_j \to 0, \qquad n \to \infty, \qquad n \equiv 0 \mod d.
	\]
	As $x_n / x_{n+d} \sim \rho^d$, we may choose a sequence $k_n$ that tends to infinity slowly enough such that
	\[
	\lim_{n \to \infty} \sup_{\substack{0 \le y \le k_n \\ y \equiv 0 \mod d}} \left | \frac{x_n}{x_{n+y}} - \rho^{y}\right | = 0, \qquad n \equiv 0 \mod d.
	\]
	Without loss of generality we may additionally assume that $k_n = o(n)$. Hence 
	\begin{align*}
	\frac{1}{c_n} \sum_{i+j=n+1} c_i c_j &= \frac{1}{x_{n-1}} \sum_{i+j=n+1} \frac{n}{ij} x_{i-1}x_{j-1}  \\
	&= o(1) + 2 \sum_{1 \le i < k_n} c_i  \frac{n}{n-i}  \frac{x_{n-i}}{x_{n-1}} ,  \\
	&\to 2 \cC(\rho) / \rho.
	\end{align*}
	as $n \equiv 1 \mod d$ becomes large. Thus, the shifted series $\cC(z)/z$ belongs to the class $\mathscr{S}_d$.
\end{proof}

\begin{proof}[Proof of Theorem~\ref{te:app2}]
	The convergence of the small fragments follows directly by Theorem~\ref{te:extension}. The asymptotic expression of $[z^n] \cA(z)$ follows from the observation that $n \equiv a \mod d$ implies
	\[
	[z^n] \cA(z) =  [z^n] (\Set_a \circ \cC)(z) \sim C_{a-1} [z^{n - (a-1)}] \cC(z),
	\]
	similar as in Equation~\eqref{eq:refme}. 
\end{proof}

\bibliographystyle{siam}
\bibliography{gibbs}

\begin{thebibliography}{10}

\bibitem{MR2032426}
{\sc R.~Arratia, A.~D. Barbour, and S.~Tavar{\'e}}, {\em Logarithmic
  combinatorial structures: a probabilistic approach}, EMS Monographs in
  Mathematics, European Mathematical Society (EMS), Z\"urich, 2003.

\bibitem{MR2121024}
{\sc A.~D. Barbour and B.~L. Granovsky}, {\em Random combinatorial structures:
  the convergent case}, J. Combin. Theory Ser. A, 109 (2005), pp.~203--220.

\bibitem{MR1763972}
{\sc J.~P. Bell, E.~A. Bender, P.~J. Cameron, and L.~B. Richmond}, {\em
  Asymptotics for the probability of connectedness and the distribution of
  number of components}, Electron. J. Combin., 7 (2000), pp.~Research Paper 33,
  22 pp. (electronic).

\bibitem{MR2383441}
{\sc E.~A. Bender, E.~R. Canfield, and L.~B. Richmond}, {\em Coefficients of
  functional compositions often grow smoothly}, Electron. J. Combin., 15
  (2008), pp.~Research Paper 21, 8.

\bibitem{2007arXiv0710.2995B}
{\sc O.~{Bernardi}, M.~{Noy}, and D.~{Welsh}}, {\em {On the growth rate of
  minor-closed classes of graphs}}, ArXiv e-prints,  (2007).

\bibitem{MR2810913}
{\sc M.~Bodirsky, {\'E}.~Fusy, M.~Kang, and S.~Vigerske}, {\em Boltzmann
  samplers, {P}\'olya theory, and cycle pointing}, SIAM J. Comput., 40 (2011),
  pp.~721--769.

\bibitem{MR0348393}
{\sc J.~Chover, P.~Ney, and S.~Wainger}, {\em Functions of probability
  measures}, J. Analyse Math., 26 (1973), pp.~255--302.

\bibitem{MR2484382}
{\sc M.~Drmota}, {\em Random trees}, SpringerWienNewYork, Vienna, 2009.
\newblock An interplay between combinatorics and probability.

\bibitem{MR714482}
{\sc P.~Embrechts}, {\em The asymptotic behaviour of series and power series
  with positive coefficients}, Med. Konink. Acad. Wetensch. Belgi\"e, 45
  (1983), pp.~41--61.

\bibitem{MR772907}
{\sc P.~Embrechts and E.~Omey}, {\em Functions of power series}, Yokohama Math.
  J., 32 (1984), pp.~77--88.

\bibitem{MR2453776}
{\sc M.~M. Erlihson and B.~L. Granovsky}, {\em Limit shapes of {G}ibbs
  distributions on the set of integer partitions: the expansive case}, Ann.
  Inst. Henri Poincar\'e Probab. Stat., 44 (2008), pp.~915--945.

\bibitem{MR2483235}
{\sc P.~Flajolet and R.~Sedgewick}, {\em Analytic combinatorics}, Cambridge
  University Press, Cambridge, 2009.

\bibitem{MR3097424}
{\sc S.~Foss, D.~Korshunov, and S.~Zachary}, {\em An introduction to
  heavy-tailed and subexponential distributions}, Springer Series in Operations
  Research and Financial Engineering, Springer, New York, second~ed., 2013.

\bibitem{MR2476775}
{\sc O.~Gim{\'e}nez and M.~Noy}, {\em Asymptotic enumeration and limit laws of
  planar graphs}, J. Amer. Math. Soc., 22 (2009), pp.~309--329.

\bibitem{MR1603725}
{\sc X.~Gourdon}, {\em Largest component in random combinatorial structures},
  in Proceedings of the 7th {C}onference on {F}ormal {P}ower {S}eries and
  {A}lgebraic {C}ombinatorics ({N}oisy-le-{G}rand, 1995), vol.~180, 1998,
  pp.~185--209.

\bibitem{MR0357214}
{\sc F.~Harary and E.~M. Palmer}, {\em Graphical enumeration}, Academic Press,
  New York-London, 1973.

\bibitem{MR2245498}
{\sc S.~Janson}, {\em Random cutting and records in deterministic and random
  trees}, Random Structures Algorithms, 29 (2006), pp.~139--179.

\bibitem{MR2908619}
\leavevmode\vrule height 2pt depth -1.6pt width 23pt, {\em Simply generated
  trees, conditioned {G}alton-{W}atson trees, random allocations and
  condensation}, Probab. Surv., 9 (2012), pp.~103--252.

\bibitem{MR633783}
{\sc A.~Joyal}, {\em Une th\'eorie combinatoire des s\'eries formelles}, Adv.
  in Math., 42 (1981), pp.~1--82.

\bibitem{MR0107292}
{\sc J.~G. Kemeny}, {\em A probability limit theorem requiring no moments},
  Proc. Amer. Math. Soc., 10 (1959), pp.~607--612.

\bibitem{MR699986}
{\sc J.~Labelle}, {\em Applications diverses de la th\'eorie combinatoire des
  esp\`eces de structures}, Ann. Sci. Math. Qu\'ebec, 7 (1983), pp.~59--94.

\bibitem{MR2418771}
{\sc C.~McDiarmid}, {\em Random graphs on surfaces}, J. Combin. Theory Ser. B,
  98 (2008), pp.~778--797.

\bibitem{MR2507738}
\leavevmode\vrule height 2pt depth -1.6pt width 23pt, {\em Random graphs from a
  minor-closed class}, Combin. Probab. Comput., 18 (2009), pp.~583--599.

\bibitem{MR2117936}
{\sc C.~McDiarmid, A.~Steger, and D.~J.~A. Welsh}, {\em Random planar graphs},
  J. Combin. Theory Ser. B, 93 (2005), pp.~187--205.

\bibitem{MR2249274}
\leavevmode\vrule height 2pt depth -1.6pt width 23pt, {\em Random graphs from
  planar and other addable classes}, in Topics in discrete mathematics, vol.~26
  of Algorithms Combin., Springer, Berlin, 2006, pp.~231--246.

\bibitem{MR1179830}
{\sc L.~R. Mutafchiev}, {\em Local limit theorems for sums of power series
  distributed random variables and for the number of components in labelled
  relational structures}, Random Structures Algorithms, 3 (1992), pp.~403--426.

\bibitem{MR2236510}
{\sc S.~Norine, P.~Seymour, R.~Thomas, and P.~Wollan}, {\em Proper minor-closed
  families are small}, J. Combin. Theory Ser. B, 96 (2006), pp.~754--757.

\bibitem{noysurvey}
{\sc M.~Noy}, {\em Random planar graphs and beyond}, Proc. ICM,  (2014).

\bibitem{MR2245368}
{\sc J.~Pitman}, {\em Combinatorial stochastic processes}, vol.~1875 of Lecture
  Notes in Mathematics, Springer-Verlag, Berlin, 2006.
\newblock Lectures from the 32nd Summer School on Probability Theory held in
  Saint-Flour, July 7--24, 2002, With a foreword by Jean Picard.

\bibitem{MR0284380}
{\sc R.~W. Robinson}, {\em Enumeration of non-separable graphs}, J.
  Combinatorial Theory, 9 (1970), pp.~327--356.

\end{thebibliography}

\end{document}